\documentclass[12pt, reqno]{amsart}

\author[P.~Leonetti]{Paolo Leonetti}
\address{Institute of Analysis and Number Theory, Graz University of Technology, Kopernikusgasse 24/II, 8010 Graz (Austria)}
\email{leonetti.paolo@gmail.com}
\urladdr{\href{https://sites.google.com/site/leonettipaolo/}{sites.google.com/site/leonettipaolo}}

\keywords{Ideal cluster points, ideal limit points, Erd{\H o}s--Ulam ideal, generalized density ideal, $F_\sigma$-ideal, summable ideal, asymptotic density, meager set.}
\subjclass[2010]{Primary: 40A35. Secondary: 11B05, 54A20.}

\title{Limit points of subsequences}

\usepackage{amsmath}
\usepackage{amssymb}
\usepackage{amsthm}
\usepackage[left=3.5cm, right=3.5cm, paperheight=11.8in]{geometry}
\usepackage{hyperref}
\usepackage{fancyhdr}
\usepackage{enumitem}
\usepackage{comment}
\usepackage{nicefrac}
\usepackage{mathrsfs}
\usepackage{graphicx}
\usepackage[utf8]{inputenc}

\AtBeginDocument{%
   \def\MR#1{}
}

\newtheorem{thm}{Theorem}[section]
\newtheorem{cor}[thm]{Corollary}
\newtheorem{lem}[thm]{Lemma}

\theoremstyle{definition} 
\newtheorem{defi}[thm]{Definition}
\let\olddefi\defi
\renewcommand{\defi}{\olddefi\normalfont}
\newtheorem{example}[thm]{Example}
\let\oldexample\example
\renewcommand{\example}{\oldexample\normalfont}
\newtheorem{rmk}[thm]{Remark}
\let\oldrmk\rmk
\renewcommand{\rmk}{\oldrmk\normalfont}

\newcommand{\clusterfin}{\mathrm{L}_x}

\hypersetup{
    pdftitle={Limit points of subsequences},
    pdfauthor={Paolo Leonetti},
    pdfmenubar=false,
    pdffitwindow=true,
    pdfstartview=FitH,
    colorlinks=true,
    linkcolor=blue,
    citecolor=green,
    urlcolor=cyan
}

\providecommand{\MR}[1]{}
\providecommand{\bysame}{\leavevmode\hbox to3em{\hrulefill}\thinspace}
\providecommand{\MR}{\relax\ifhmode\unskip\space\fi MR }

\providecommand{\href}[2]{#2}
                                 
\begin{document}

\maketitle
\thispagestyle{empty}

\begin{abstract}
\noindent  Let $x$ be a sequence taking values in a separable metric space and let $\mathcal{I}$ be an $F_{\sigma\delta}$-ideal on the positive integers (in particular, $\mathcal{I}$ can be any Erd{\H o}s--Ulam ideal or any summable ideal). 
It is shown that the collection of subsequences of $x$ which preserve the set of $\mathcal{I}$-cluster points of $x$ is of second category if and only if the set of $\mathcal{I}$-cluster points of $x$ coincides with the set of ordinary limit points of $x$; moreover, in this case, it is comeager. 

The analogue for $\mathcal{I}$-limit points is provided. As a consequence, the collection of subsequences of $x$ which preserve the set of ordinary limit points is comeager. 
\end{abstract}

\section{Introduction}\label{sec:intro}
Let $x$ be a real sequence. By a classical result of Buck \cite{MR0009997}, the set of ordinary limit points of ``almost every'' subsequence of $x$ coincides with the set of ordinary limit points of the original sequence, in the sense of Lebesgue measure. 
In the same direction, it has been recently shown in \cite{Leo17b, Leo17} that almost all subsequences preserve the set of statistical cluster points of $x$ [resp., statistical limit points], see details below. 

The aim of this article is to provide their topological analogues, obtaining another example of the ``duality'' between measure and category. 
In particular, our main results (Theorems \ref{thm:main} and \ref{thm:main2} in Section \ref{sec:main}) imply that the set of subsequences considered by Buck \cite{MR0009997} is always comeager. 
In addition, they show that the set of subsequences of $x$ which preserve the statistical cluster points [resp., statistical limit points] is meager if and only if there exists an ordinary limit point of $x$ which is not a statistical cluster point of $x$ [resp., statistical limit point].  

First, we recall some definitions. Let $\mathcal{I}\subseteq \mathcal{P}(\mathbf{N})$ be an ideal, that is, a family of subsets of positive integers closed under taking finite unions and subsets. It is also assumed that $\mathbf{N} \notin \mathcal{I}$ and that $\mathcal{I}$ contains the collection $\mathrm{Fin}$ of finite subsets. Note that the family of $\alpha$-density zero sets 
\begin{equation}\label{eq:Ialpha}
\textstyle \mathcal{I}_\alpha:=\left\{A\subseteq \mathbf{N}:\, \sum_{i \in A \cap [1,n]}i^\alpha=o\left(\sum_{i \in [1,n]}i^\alpha\right)\,\text{ as }n\to \infty\right\}
\end{equation}
is an ideal for each real parameter $\alpha \ge -1$ (as it has been remarked by the anonymous referee, $\mathcal{I}_\alpha=\mathcal{I}_0$ for each $\alpha>-1$, see \cite[Corollary 2]{Sle}). 

Thus, given a sequence $x=(x_n)$ taking values in a topological space $X$, we denote by $\Gamma_x(\mathcal{I})$ the set of $\mathcal{I}$\emph{-cluster points} of $x$, i.e., the set of all $\ell \in X$ such that $\{n: x_n \in U\} \notin \mathcal{I}$ for all neighborhoods $U$ of $\ell$.  
Moreover, we denote by $\Lambda_x(\mathcal{I})$ the set of $\mathcal{I}$\emph{-limit points} of $x$, i.e., the set of all ordinary limit points $\ell \in X$ of subsequences $(x_{n_k})$ such that $\{n_k: k \in \mathbf{N}\} \notin \mathcal{I}$. 
Hereafter, we shorten the set of ordinary limit points with $\clusterfin:=\Lambda_x(\mathrm{Fin})$, which coincides with $\Gamma_x(\mathrm{Fin})$ if $X$ is first countable.  
It is well known and easily seen that $\Lambda_x(\mathcal{I})\subseteq \Gamma_x(\mathcal{I}) \subseteq \clusterfin$ and that $\Gamma_x(\mathcal{I})$ is closed, cf. e.g. \cite{LMxyz}. 

Statistical cluster points and statistical limit points (that is, $\mathcal{I}_{0}$-cluster points and $\mathcal{I}_0$-limit points) of real sequences were introduced by Fridy \cite{MR1181163}, cf. also \cite{PaoloMarek17, MR1372186, MR2463821, MR1416085, MR1838788, Leo17, Leo17b, LeoJMAA19}.  
It is worth noting that ideal cluster points have been studied much before under a different name. Indeed, as it follows by \cite[Theorem 4.2]{LMxyz}, they correspond to classical ``cluster points'' of a filter $\mathscr{F}$ on $\mathbf{R}$ (depending on $x$), cf. \cite[Definition 2, p.69]{MR1726779}.

At this point, consider the natural bijection between the collection of all subsequences $(x_{n_k})$ of $(x_n)$ and real numbers $\omega \in (0,1]$ with non-terminating dyadic expansion $\sum_{i\ge 1}d_i(\omega)2^{-i}$, where $d_i(\omega)=1$ if $i=n_k$, for some integer $k$, and $d_i(\omega)=0$ otherwise, 
cf. \cite{MR1260176}.  
Accordingly, for each $\omega \in (0,1]$, denote by $x \upharpoonright \omega$ the subsequence of $(x_n)$ obtained by omitting $x_i$ if and only if $d_i(\omega)=0$. 
In other words, denoting by $(n_k)$ the increasing sequence of all $i \in \mathbf{N}$ such that $d_i(\omega)=1$, then $x \upharpoonright \omega$ stands for the subsequence $(x_{n_k})$. 
(This should not be confused with the notion of \textquotedblleft nonthin subsequence\textquotedblright used, e.g., in \cite{MR1181163} where it is required, additionally, that $\{n_k: k \in \mathbf{N}\} \notin \mathcal{I}$.)

Finally, let $\lambda: \mathscr{M}\to \mathbf{R}$ denote the Lebesgue measure, where $\mathscr{M}$ stands for the completion of the Borel $\sigma$-algebra on $(0,1]$. 

As a consequence of the main results in \cite{Leo17, Leo17b} and \cite[Corollary 2]{Sle}, the following holds:
\begin{thm}\label{thm:oldmeasure}
Fix a real $\alpha \ge -1$ and let $x$ be a sequence taking values in a first countable space where all closed sets are separable. Then 
$$
\lambda\left(\left\{\omega \in (0,1]: \Gamma_x(\mathcal{I}_\alpha)=\Gamma_{x\upharpoonright \omega}(\mathcal{I}_\alpha)\right\}\right)=1
$$ 
and
$$
\lambda\left(\left\{\omega \in (0,1]: \Lambda_x(\mathcal{I}_\alpha)=\Lambda_{x\upharpoonright \omega}(\mathcal{I}_\alpha)\right\}\right)=1. 
$$
\end{thm}

The key observation in the proof of the above result is that the set of normal numbers $\Omega:=\left\{\omega \in (0,1]: \frac{1}{n}\sum_{i\le n}d_i(\omega)\to \frac{1}{2}\text{ as }n\to \infty\right\}$ has full Lebesgue measure, i.e., $\lambda(\Omega)=1$. 
Related results were obtained in \cite{MR3568092, MR0316930, MR1260176}.

On the other hand, it is well known that $\Omega$ is a meager set, that is, a set of first category. 
This suggests that the category analogue of Theorem \ref{thm:oldmeasure} does not hold in general. 
In the next section, our 
main results 
show that this is indeed the case.

\section{Preliminaries and Main Results}\label{sec:main}

We recall that an ideal $\mathcal{I}$ is said to be a \emph{P-ideal} if it is $\sigma$-directed modulo finite sets, i.e., for every sequence $(A_n)$ of sets in $\mathcal{I}$ there exists $A \in \mathcal{I}$ such that $A_n\setminus A$ is finite for all $n$. 
Moreover, by identifying sets of integers with their characteristic functions, we equip $\mathcal{P}(\mathbf{N})$ with the Cantor-space topology and therefore we can assign the topological complexity to the ideals on $\mathbf{N}$. 

A function $\varphi: \mathcal{P}(\mathbf{N}) \to [0,\infty]$ is said to be a \emph{submeasure} provided that $\varphi(\emptyset)=0$, $\varphi(\{n\})<\infty$ for all $n$, it is monotone (i.e., $\varphi(A) \le \varphi(B)$ for all $A\subseteq B$), and subadditive (i.e., $\varphi(A\cup B) \le \varphi(A)+\varphi(B)$). A submeasure $\varphi$ is \emph{lower semicontinuous} provided that $\varphi(A)=\lim_{n}\varphi(A\cap [1,n])$ for all $A$. 
By a classical result of Solecki, an ideal $\mathcal{I}$ is an analytic P-ideal if and only if there exists a lower semicontinuous submeasure $\varphi$ such that 
$$
\mathcal{I}=\mathrm{Exh}(\varphi):=\{A: \lim_{n\to \infty}\varphi(A \setminus [1,n])=0\}, 
$$ 
see \cite[Theorem 3.1]{MR1708146}. 

At this point, let $(I_n)$ be a partition of $\mathbf{N}$ in non-empty finite sets and $\mu=(\mu_n)$ be a sequence of submeasures such that each $\mu_n$ concentrates on $I_n$ and $\limsup_n \mu_n(I_n)\neq 0$. 
Then, the \emph{generalized density ideal} 
\begin{equation}\label{eq:Zmu}
\mathcal{Z}_\mu:=\{A\subseteq \mathbf{N}: \lim_{n\to \infty} \mu_n(A \cap I_n)=0\}
\end{equation}
is an analytic P-ideal: indeed, it is easy to check that $\mathcal{Z}_\mu=\mathrm{Exh}(\varphi_\mu)$, where $\varphi_\mu:=\sup_k \mu_k$.  
The class of generalized density ideals has been introduced by Farah in \cite[Section 2.10]{MR1988247}, see also \cite{MR2254542}. 
In particular, each $\mathcal{Z}_\mu$ is an $F_{\sigma \delta}$-ideal.

It is worth noting that generalized density ideals have been used in different contexts, see e.g. \cite{MR3436368, MR2320288}, and it is a very rich class. Indeed, if each $\mu_n$ is a measure then $\mathcal{Z}_\mu$ is a density ideal, as defined in \cite[Section 1.13]{MR1711328}. In particular, it includes $\emptyset \times \mathrm{Fin}$ and also the Erd{\H o}s--Ulam ideals introduced by Just and Krawczyk in \cite{MR748847}, i.e., ideals of the type $\mathrm{Exh}(\varphi_f)$ where $f: \mathbf{N} \to (0,\infty)$ is a function such that $\sum_{n \in \mathbf{N}}f(n)=\infty$ and $f(n)=o\left(\sum_{i\le n}f(i)\right)$ as $n\to \infty$ and $\varphi_f: \mathcal{P}(\mathbf{N}) \to (0,\infty)$ is the submeasure defined by
$$
\varphi_f(A)=\sup_{n \in \mathbf{N}} \frac{\sum_{i\le n,\, i \in A}f(i)}{\sum_{i\le n}f(i)},
$$
see \cite[pp. 42--43]{MR1711328}. 
In addition, it contains the ideals associated with suitable modifications of the natural density, the so-called simple density ideals, see \cite{MR3391516}. 
Lastly, a large class of generalized density ideals has been defined by Louveau and Veli\v{c}kovi\'{c} in \cite{MR1169042}, cf. also \cite[Section 2.11]{MR1988247}.

Note that also the class of $F_\sigma$-ideals is quite large: it contains, among others, all the summable ideals (i.e., P-ideals of the form $\{A: \sum_{n \in A}f(n)<\infty\}$, where $f: \mathbf{N} \to [0,\infty)$ is a function such that $\sum_{n \in \mathbf{N}}f(n)=\infty$, see \cite[Section 1.12]{MR1711328}), finitely generated ideals $\{A: A\setminus B \in \mathrm{Fin}\}$ for some infinite set $B$ as in \cite[Example 2]{MR3568092}, fragmented ideals \cite{MR3171606}, and Tsirelson ideals defined in \cite{MR1680630, MR1680626}; 
in addition, it has been shown in \cite[Section 1.11]{MR1711328} that there exists an $F_\sigma$ P-ideal which is not summable. 

Finally, we say that an ideal $\mathcal{I}$ is $F_\sigma$\emph{-separated from its dual filter} $\mathcal{I}^\star:=\{S\subseteq \mathbf{N}: S^c \in \mathcal{I}\}$ whenever there exists an $F_\sigma$-set $A \subseteq \mathcal{P}(\mathbf{N})$ such that $\mathcal{I}\subseteq A$ and $A \cap \mathcal{I}^\star=\emptyset$. 
The family of these ideals includes all $F_{\sigma\delta}$-ideals, see \cite[Corollary 1.5]{MR1758325}. 
Moreover, it is known that a Borel ideal $\mathcal{I}$ is $F_\sigma$-separated from its dual filter if and only if it does not contain an isomorphic copy of the $\mathrm{Fin}\times \mathrm{Fin}$, see  \cite[Theorem 4]{MR2491780} for details.

Our first main result (about $\mathcal{I}$-cluster points) follows:
\begin{thm}\label{thm:main}
Let $x$ be a sequence taking values in a first countable space $X$ where all closed sets are separable and let $\mathcal{I}$ be an ideal which is $F_\sigma$-separated from its dual filter \textup{(}in particular, any $F_{\sigma\delta}$-ideal\textup{)}.  
Then
\begin{equation}\label{eq:claimedset}
\left\{\omega \in (0,1]: \Gamma_{x \upharpoonright \omega}(\mathcal{I})=\Gamma_x(\mathcal{I})\right\}
\end{equation}
is not meager if and only if $\Gamma_x(\mathcal{I})=\clusterfin$. Moreover, in this case, it is comeager.
\end{thm}

Since the ideal of finite sets $\mathrm{Fin}$ is countable (hence, an $F_\sigma$-ideal), we obtain the topological analogue of Buck's result \cite{MR0009997}:
\begin{cor}\label{cor:fin}
Let $x$ be a sequence as in Theorem \ref{thm:main}. Then 
the set of subsequences which preserve the ordinary limit points of $x$
is comeager.
\end{cor}

Then, we have also the analogue of Theorem \ref{thm:main} for $\mathcal{I}$-limit points:
\begin{thm}\label{thm:main2}
Let $x$ be a sequence taking values in a first countable space $X$ where all closed sets are separable and let $\mathcal{I}$ be a generalized density ideal or an $F_\sigma$-ideal. 
Then
\begin{equation}\label{eq:claimedset2}
\left\{\omega \in (0,1]: \Lambda_{x \upharpoonright \omega}(\mathcal{I})=\Lambda_x(\mathcal{I})\right\}
\end{equation}
is not meager if and only if $\Lambda_x(\mathcal{I})=\clusterfin$. Moreover, in this case, it is comeager.
\end{thm}

Recalling that Erd{\H o}s--Ulam ideals are density ideals (hence, in particular, generalized density ideals), the following corollaries are immediate (we omit details):
\begin{cor}\label{cor:111}
Let $x$ be a sequence taking values in a separable metric space $X$ and let $\mathcal{I}$ be an Erd{\H o}s--Ulam ideal. Then 
the set \eqref{eq:claimedset} \textup{[}resp., the set \eqref{eq:claimedset2}\textup{]} is not meager if and only if $\Gamma_x(\mathcal{I})=\clusterfin$ \textup{[}resp., $\Lambda_x(\mathcal{I})=\clusterfin$\textup{]}. 
\end{cor}
In this regard, for each $\alpha\ge -1$, the ideal $\mathcal{I}_\alpha$ defined in \eqref{eq:Ialpha} is an Erd{\H o}s--Ulam ideal. In particular, setting $\alpha=0$ and $X=\mathbf{R}$, we obtain the main result given in \cite{LMM}:

\begin{cor}\label{cor:realcluster}
Let $x$ be a real sequence. Then the set of its subsequences which preserve the statistical cluster points \textup{[}resp., statistical limit points\textup{]} of $x$ is comeager if and only if it is not meager if and only if every ordinary limit point of $x$ is also a statistical cluster point \textup{[}resp., statistical limit point\textup{]} of $x$.
\end{cor}

\section{Proof of Theorem \ref{thm:main}}

We start an easy preliminary observation:
\begin{lem}\label{lemma_inclusion}
Let $x$ be a sequence taking values in a first countable space and let $\mathcal{I}$ be an ideal. Then $\Lambda_{x \upharpoonright \omega}(\mathcal{I})\subseteq \clusterfin$ and $\Gamma_{x \upharpoonright \omega}(\mathcal{I})\subseteq \clusterfin$ for each $\omega \in (0,1]$. 
\end{lem}
\begin{proof}
It follows by $\Lambda_{x \upharpoonright \omega}(\mathcal{I}) \subseteq \Gamma_{x \upharpoonright \omega}(\mathcal{I}) \subseteq \mathrm{L}_{x \upharpoonright \omega} \subseteq \clusterfin$ for all $\omega \in (0,1]$.
\end{proof}

\begin{lem}\label{lem:fsigma}
Let $x$ be a sequence in a first countable space $X$ and let $\mathcal{I}$ be an ideal which is $F_{\sigma}$-separated from its dual filter. 
Then $\left\{\omega \in (0,1]: \ell \in \Gamma_{x \upharpoonright \omega}(\mathcal{I})\right\}$ is comeager for every $\ell \in\clusterfin$.
\end{lem}
\begin{proof}
If $\mathrm{L}_x=\emptyset$ there is nothing to prove. Otherwise, fix $\ell \in \clusterfin$ and let $(U_m)$ be a decreasing local base at $\ell$. Let us suppose that the ideal $\mathcal{I}$ is $F_{\sigma}$-separated from its dual filter through the set $A:=\bigcup_n A_n \subseteq \mathcal{P}(\mathbf{N})$, where each $A_n$ is closed.  Then we need to show that $S:=\{\omega \in (0,1]: \ell \notin \Gamma_{x \upharpoonright \omega}(\mathcal{I})\}$ is meager. Note that $S\subseteq \bigcup_{m\ge 1} \bigcup_{k\ge 1}S_{m,k}$, where
$$
S_{m,k}:=\left\{\omega \in (0,1]: \left\{n: (x \upharpoonright \omega)_n \in U_m\right\} \in A_k\right\}
$$
for all $m,k \in \mathbf{N}$. It is sufficient to show that each $S_{m,k}$ is nowhere dense.

We show that $S_{m,k}$ is closed. Fix $\omega_0 \in S_{m,k}^c$ (if there is no such $\omega_0$ then $S_{m,k}=(0,1]$ is closed in $(0,1]$). Since $A_k$ is closed, there exists $n_0 \in \mathbf{N}$ such that
$$
\{\omega \in (0,1]: d_n(\omega)=d_n(\omega) \text{ for all }n\le n_0\} \subseteq S_{m,k}^c.
$$
Hence $S_{m,k}$ is closed.

Finally we need to show that $S_{m,k}$ contains no non-empty open sets. Fix $\omega_1 \in (0,1]$ such that the subsequence $x \upharpoonright \omega_1$ converges to $\ell$ and let us suppose for the sake of contradiction that there exist $e_1,\ldots,e_{n_1} \in \{0,1\}$ such that $\omega \in S_{m,k}$ whenever $d_n(\omega)=e_n$ for all $n\le n_1$. Define 
$$
d_n(\omega^\star)=\begin{cases}
\text{ }e_n \hspace{4mm}& \text{ for }n\le n_1,\\
\text{ }d_n(\omega_1) &\text{ for }n > n_1.
\end{cases}
$$ 
Then $\omega^\star \in S_{m,k}$ and, on the other hand, the subsequence $x\upharpoonright \omega^\star$ converges to $\ell$, and thus $\{n: (x\upharpoonright \omega^\star)_n \in U_m\}\in \mathcal{I}^\star$, which gives the desired contradiction.
\end{proof}

Let us finally prove Theorem \ref{thm:main}. 
\begin{proof}
[Proof of Theorem \ref{thm:main}]
\textsc{If Part.} Let us suppose that $\Gamma_x(\mathcal{I})=\clusterfin$. Hence, it is claimed that the set $\left\{\omega \in (0,1]: \Gamma_{x \upharpoonright \omega}(\mathcal{I})=\clusterfin\right\}$ is comeager.

If $\clusterfin=\emptyset$, then the claim follows by Lemma \ref{lemma_inclusion}. Hence, let us suppose hereafter that $\clusterfin$ is non-empty. Since $\clusterfin$ is closed, 
there exists a non-empty countable set $\mathscr{L}$ whose closure is $\clusterfin$. Moreover, since the collection of meager sets is a $\sigma$-ideal, we get by Lemma \ref{lem:fsigma} that
$$
\mathcal{M}:=\left\{\omega \in (0,1]: \ell \notin \Gamma_{x \upharpoonright \omega}(\mathcal{I})\,\text{ for some }\,\ell \in \mathscr{L}\right\}
$$
is meager. Hence $\mathscr{L}\subseteq \Gamma_{x \upharpoonright \omega}(\mathcal{I})$ for each $\omega \in \mathcal{M}^c:=(0,1]\setminus \mathcal{M}$. 
At this point, fix $\omega \in \mathcal{M}^c$. It follows 
that $\Gamma_{x \upharpoonright \omega}(\mathcal{I})$ contains also the closure of $\mathscr{L}$, i.e., $\clusterfin$. On the other hand, $\Gamma_{x \upharpoonright \omega}(\mathcal{I})\subseteq \clusterfin$ by Lemma \ref{lemma_inclusion}. Therefore $\Gamma_{x \upharpoonright \omega}(\mathcal{I})=\clusterfin$ for each $\omega \in \mathcal{M}^c$.

\vspace{4mm}

\textsc{Only If Part.} Let us suppose that $\Gamma_x(\mathcal{I}) \neq \clusterfin$ so that 
there exists a point $\ell \in \clusterfin\setminus \Gamma_x(\mathcal{I})$. Therefore, the set of all $\omega \in (0,1]$ such that $\Gamma_{x \upharpoonright \omega}(\mathcal{I})=\Gamma_x(\mathcal{I})$ is contained in $\left\{\omega \in (0,1]: \ell \notin \Gamma_{x \upharpoonright \omega}(\mathcal{I})\right\}$ which, thanks to Lemma \ref{lem:fsigma}, is a meager set.
\end{proof}

\section{Proof of Theorem \ref{thm:main2}}
We proceed with some technical lemmas (for the case of generalized density ideals):
\begin{lem}\label{lemma_key_limit}
Let $x$ be a sequence taking values in a first countable space $X$, let $\mathcal{I}$ be a generalized density ideal such that $\mathcal{I}=\mathcal{Z}_\mu$ as in \eqref{eq:Zmu}, and fix 
$q \in (0,\limsup_{n\to \infty}\mu_n(I_n))$. Fix also $\ell \in X$ and a decreasing local base $(U_m)$ at $\ell$. Then, the set
$$
\mathscr{V}_\ell=\mathscr{V}_\ell(x; q):=\left\{\omega \in (0,1]: \limsup_{n\to \infty}\mu_n(A_{\omega,m} \cap I_n)\ge q \text{ for all }m\right\},
$$
where $A_{\omega,m}:=\{k: x_{n_k} \in U_m\}$ and $(x_{n_k}):=x\upharpoonright \omega$, is either comeager or empty.
\end{lem}
\begin{proof}
Let us suppose $\mathscr{V}_\ell$ is non-empty, so that, in particular, $\ell \in \clusterfin$. Then, it is claimed that $\mathscr{V}_\ell^c$ is meager. 
%
For each $m,n \in \mathbf{N}$ and $\omega \in (0,1]$ set also $\nu_{\omega,m}(n):=\mu_n(A_{\omega,m}\cap I_n)$ to ease the notation. It follows that 
\begin{displaymath}
\begin{split}
\mathscr{V}_\ell^c &\textstyle =\left(\bigcap_{m\ge 1}\bigcap_{j\ge 1}\left\{\omega: \nu_{\omega,m}(n) \ge q\left(1-2^{-j}\right) \text{ for infinitely many }n\right\}\right)^c \\
&\textstyle =\bigcup_{m\ge 1}\bigcup_{j\ge 1}\left\{\omega: \nu_{\omega,m}(n) < q\left(1-2^{-j}\right) \text{ for all sufficiently large }n\right\}\\
&\textstyle =\bigcup_{m\ge 1}\bigcup_{j\ge 1}\bigcup_{t\ge 1} \bigcap_{s\ge t}\left\{\omega: \nu_{\omega,m}(s) < q\left(1-2^{-j}\right)\right\}.
\end{split}
\end{displaymath}
Hence, it is sufficient to show that, for every $q\in (0,\limsup_{n\to \infty}\mu_n(I_n))$, each set 
$
\textstyle B_{m,t}:=\bigcap_{s\ge t}\,\left\{\omega: \nu_{\omega,m}(s) < q\right\}
$ 
is nowhere dense: indeed, this would imply that $\mathscr{V}_\ell$ is comeager.

Equivalently, let us prove that, for each fixed $m,t \in \mathbf{N}$, every non-empty open interval $(a,b)\subseteq (0,1)$ contains a non-empty interval disjoint to $B_{m,t}$. Fix $\omega_0 \in (a,b)$ with finite dyadic representation $\sum_{i=1}^r 2^{-h_i}$ such that $\omega_0+2^{-h_r}<b$. Moreover, since $\ell \in \clusterfin$, there exists $\omega_1 \in (0,1]$ such that $x \upharpoonright \omega_1 \to \ell$, hence 
$$
\limsup_{s\to \infty}\nu_{\omega_1,m}(s)=\limsup_{n\to \infty}\mu_n(I_n)>0.
$$
It follows that there exists an integer $s_\star>\max(t,h_r)$ such that $d_{s_\star}(\omega_1)=1$ and 
$
\nu_{\omega_\star,m}(s_\star) \ge q,
$ 
where $\omega_\star:=\omega_0+\sum_{h_r<i\le s_\star}d_i(\omega_1)/2^i$. Therefore, each $\omega \in \left(\omega_\star,\omega_\star+2^{-s_\star}\right)$ starts with the same binary representation of $\omega_\star$, so that $\nu_{\omega,m}(s_\star) \ge q$ and, in particular, does not belong to $B_{m,t}$. This concludes the proof since $\left(\omega_\star,\omega_\star+2^{-s_\star}\right) \subseteq \left(\omega_0,\omega_0+2^{-h_r}\right)$ which, in turn, is contained in $(a,b)$.
\end{proof}

\begin{lem}\label{lem:identity}
With the same notation of Lemma \ref{lemma_key_limit}, it holds
$$
\textstyle \left\{\omega \in (0,1]: \ell \in \Lambda_{x \upharpoonright \omega}(\mathcal{I})\right\}=\bigcup_{q>0} \mathscr{V}_\ell(x; q).
$$
\end{lem}
\begin{proof}
Let us fix $\omega \in (0,1]$ such that $\ell \in \Lambda_{x\upharpoonright \omega}(\mathcal{I})$, i.e., there exist $\eta \in (0,1]$ and $q>0$ such that the subsequence $(x\upharpoonright \omega) \upharpoonright \eta \to \ell$ and $\limsup_{j}\mu_j(\{k_t: t \in \mathbf{N}\}\cap I_j) \ge q$, where we denote by $(x_{n_k})$ and $(x_{n_{k_t}})$ the subsequences $x\upharpoonright \omega$ and $(x\upharpoonright \omega) \upharpoonright \eta$, respectively. Then, for each $m \in \mathbf{N}$ there is a finite set $F \in \mathrm{Fin}$ such that
\begin{displaymath}
\begin{split}
\limsup_{j\to \infty}\mu_j(\{k: x_{n_k} \in U_m\}\cap I_j) &\ge \limsup_{j\to \infty}\mu_j(\{k_t: x_{n_{k_t}} \in U_m \text{ and }t \in \mathbf{N}\}\cap I_j)\\
&= \limsup_{j\to \infty}\mu_j((\{k_t: t \in \mathbf{N}\}\setminus F)\cap I_j)\\
&= \limsup_{j\to \infty}\mu_j(\{k_t: t \in \mathbf{N}\}\cap I_j) \ge q.
\end{split}
\end{displaymath}
This implies that $\omega \in \mathscr{V}_\ell(x;q)$.

Conversely, let us fix $\omega \in \mathscr{V}_\ell(x;q)$ for some $q>0$, that is, $\limsup_{j}\mu_j(\{k: x_{n_k} \in U_m\}\cap I_j) \ge q$ for all $m$. Hence, for each $m \in \mathbf{N}$, there exists an increasing sequence $(j_{m,r})$ of positive integers such that
$$
\textstyle \mu_{j_{m,r}}\left(\{k: x_{n_k} \in U_m\}\cap I_{j_{m,r}}\right)\ge q\left(1-\frac{1}{2^r}\right)
$$
for all $r$. Define the increasing sequence $(r_m)$ of positive integers such that $r_1:=1$ and, recursively, $r_{m+1}$ is the smallest integer $r>r_m$ for which $j_{m+1,r}>j_{m,r_m}$. 
At this point, define the subsequence $(x_{n_{k_t}})$ of $(x_{n_k})$ by picking the index $k$ if and only if there exists $m \in \mathbf{N}$ for which $x_{n_k} \in U_{j_{m,r_m}}$ and $k \in I_{j_{m,r_m}}$. 
It follows by construction that the subsequence $(x_{n_{k_t}})$ is convergent to $\ell$ and that
\begin{displaymath}
\textstyle \mu_{j_{m,r_m}}\left(\{k_t: t \in \mathbf{N}\}\right)=\mu_{j_{m,r_m}}\left(\{k: x_{n_k} \in U_{j_{m,r_m}}\}\cap I_{j_{m,r_m}} \right) \ge q\left(1-\frac{1}{2^{r_m}}\right) 
\end{displaymath}
for all $m \in \mathbf{N}$, that is, $\{k_t: t \in \mathbf{N}\} \notin \mathcal{Z}_\mu$ and $\ell \in \Lambda_{x \upharpoonright \omega}(\mathcal{I})$.
\end{proof}


\begin{cor}\label{cor:limit123}
Let $x$ be a sequence taking values in a first countable space $X$ and let $\mathcal{I}$ be a generalized density ideal. Then $\left\{\omega \in (0,1]: \ell \in \Lambda_{x \upharpoonright \omega}(\mathcal{I})\right\}$ is comeager for every $\ell \in\clusterfin$.
\end{cor}
\begin{proof}
Fix $\ell \in \clusterfin$, otherwise there is nothing to prove. Then, there exists $\omega_0 \in (0,1]$ such that $x\upharpoonright \omega_0 \to \ell$. Hence, given $q_0 \in (0,\limsup_{n\to \infty}\mu_n(I_n))$, the set $\mathscr{V}_\ell(x; q_0)$ contains $\omega_0$; in particular, it is non-empty and, thanks to Lemma \ref{lemma_key_limit}, it is comeager. Lastly, the claim follows by the fact that, thanks to Lemma \ref{lem:identity}, the inclusion $\mathscr{V}_\ell(x; q_0)\subseteq \left\{\omega \in (0,1]: \ell \in \Lambda_{x \upharpoonright \omega}(\mathcal{I})\right\}$ holds.
\end{proof}

Lastly, we show that a certain subset of $\mathcal{I}$-limit points $\ell \in X$ is closed.
\begin{lem}\label{lem:closedness}
With the same notation of Lemma \ref{lemma_key_limit}, the set
$$
\textstyle \Lambda_x(\mathcal{I};q):=\left\{\ell \in X: \limsup_{j\to \infty}\mu_j\left(\{n: x_n \in U_m\}\right)\ge q \text{ for all }m\right\}
$$
is closed for each $q \in \mathbf{R}$. 
\end{lem}
\begin{proof}
Equivalently, we have to prove that the set
$$
\textstyle G:=\left\{\ell \in X: \limsup_{j\to \infty}\mu_j\left(\{n: x_n \in U_m\}\right)< q \text{ for some }m\right\}
$$
is open for each $q$. This is obvious if $G$ is empty. Otherwise, let us fix $\ell \in G$ and let $(U_m)$ be a decreasing local base at $\ell$. Then, there exists $m_0 \in \mathbf{N}$ such that $\limsup_j \mu_j(\{n: x_n \in U_m\} \cap I_j)<q$ for all $m\ge m_0$. Fix $\ell^\prime \in U_{m_0}$ and let $(V_m)$ a decreasing local base at $\ell^\prime$. Fix also $m_1 \in \mathbf{N}$ such that $V_{m_1} \subseteq U_{m_0}$. It follows by monotonicity that
$$
\limsup_{j\to \infty}\mu_j(\{n: x_n \in V_m\}\cap I_j) \le \limsup_{j\to \infty}\mu_j(\{n: x_n \in U_{m_0}\}\cap I_j)<q
$$
for every $m\ge m_1$. In particular, since $\ell^\prime$ has been arbitrarily fixed, $U_{m_0} \subseteq G$.
\end{proof}

\begin{proof}
[Proof of Theorem \ref{thm:main2}]
If $\mathcal{I}$ is an $F_\sigma$-ideal, then the claim follows by Theorem \ref{thm:main}. Indeed, thanks to \cite[Theorem 2.3]{PaoloMarek17}, we have $\Lambda_{x\upharpoonright \omega}(\mathcal{I})=\Gamma_{x\upharpoonright \omega}(\mathcal{I})$ for all $\omega \in (0,1]$. Hence, let us assume hereafter that $\mathcal{I}$ is a generalized density ideal.

\vspace{4mm}

\textsc{If Part.} With the same notation of the above proof, suppose that $\Lambda_x(\mathcal{I})=\clusterfin$ and, similarly, assume that $\clusterfin\neq \emptyset$. For each $\ell \in \clusterfin$, there exists $\omega_\ell \in (0,1]$ such that $x \upharpoonright \omega_\ell \to \ell$ and, in particular, $\ell \in \Lambda_{x \upharpoonright \omega_\ell}(\mathcal{I})$. Hence, for each fixed $q \in (0,\limsup_{n\to \infty}\mu_n(I_n))$, the set $\left\{\omega: \ell \in \Lambda_{x\upharpoonright \omega}(\mathcal{I};q)\right\}$ is non-empty. Moreover, note that
$$
\mathscr{V}_\ell(x;q)=\left\{\omega \in (0,1]: \ell \in \Lambda_{x\upharpoonright \omega}(\mathcal{I};q)\right\}.
$$
Thus, it follows by 
Corollary \ref{cor:limit123} 
that $\left\{\omega: \ell \notin \Lambda_{x\upharpoonright \omega}(\mathcal{I};q)\right\}$ is meager. Therefore, denoting by $\mathscr{L}$ a non-empty countable set with closure $\clusterfin$, we obtain that also $\mathcal{N}:=\left\{\omega: \ell \notin \Lambda_{x\upharpoonright \omega}(\mathcal{I};q) \text{ for some }\ell \in \mathscr{L}\right\}$ is meager, that is,
$$
\textstyle \mathcal{N}^c=\left\{\omega\in (0,1]: \mathscr{L}\subseteq \Lambda_{x\upharpoonright \omega}(\mathcal{I};q)\right\}
$$
is comeager. At this point, for each $\omega \in \mathcal{N}^c$, it follows 
by Lemma \ref{lem:closedness} 
that $\Lambda_{x\upharpoonright \omega}(\mathcal{I};q)$ contains also the closure of $\mathscr{L}$, i.e., $\clusterfin$. On the other hand, $\Lambda_{x\upharpoonright \omega}(\mathcal{I};q)\subseteq \Lambda_{x \upharpoonright \omega}(\mathcal{I})\subseteq \clusterfin$ by Lemma \ref{lemma_inclusion}. Therefore $\Lambda_{x \upharpoonright \omega}(\mathcal{I})=\clusterfin$ for each $\omega \in \mathcal{N}^c$.

\vspace{4mm}

\textsc{Only If Part.} This part goes verbatim as in the \emph{only if} part of the proof of Theorem \ref{thm:main} (using Corollary \ref{cor:limit123}).
\end{proof}

\section{Concluding Remarks}
It follows by Theorem \ref{thm:oldmeasure} that, for each $\alpha \ge -1$, $\left\{\omega: \Lambda_{x\upharpoonright \omega}(\mathcal{I}_\alpha)=\Gamma_{x\upharpoonright \omega}(\mathcal{I}_\alpha)\right\}$ has full Lebesgue measure if and only if $\Lambda_x(\mathcal{I}_\alpha)=\Gamma_x(\mathcal{I}_\alpha)$, cf. also \cite[Corollaries 2.4 and 4.5]{Leo17b}. 
On the other hand, its topological analogue is quite different. Indeed, we conclude with the following corollary, which follows from the proofs of the main results:
\begin{cor}\label{cor:final}
With the same hypotheses of Theorem \ref{thm:main2}, the sets
$$
\{\omega \in (0,1]: \Gamma_{x\upharpoonright \omega}(\mathcal{I})=\clusterfin\}\,\,\,\text{ and }\,\,\,\{\omega \in (0,1]: \Lambda_{x\upharpoonright \omega}(\mathcal{I})=\clusterfin\}
$$
are comeager. In particular, the set $\{\omega \in (0,1]: \Gamma_{x\upharpoonright \omega}(\mathcal{I})=\Lambda_{x\upharpoonright \omega}(\mathcal{I})\}$ is comeager.
\end{cor}

We leave as an open question to check whether Theorem \ref{thm:main2} may be extended to the whole class of $F_{\sigma \delta}$-ideals (hence, in particular, analytic P-ideals).

\subsection*{Acknowledgments}
The author is grateful to the anonymous referee for pointing out \cite[Corollary 2]{Sle} and for a careful reading of the manuscript which led, in particular, to an improvement of Lemma \ref{lem:fsigma} and the extension of Theorem \ref{thm:main} to all $F_{\sigma\delta}$-ideals.


\end{document}